\documentclass[11pt,reqno]{amsart}
\usepackage{enumerate, latexsym, amsmath, amsfonts, amssymb, amsthm, color}
\def\pmod #1{\ ({\rm{mod}}\ #1)}
\def\Z{\mathbb Z}
\def\N{\mathbb N}

\def\l{\left}
\def\r{\right}
\def\bg{\bigg}
\def\({\bg(}
\def\){\bg)}
\def\t{\text}
\def\f{\frac}

\def\ls{\leqslant}
\def\gs{\geqslant}
\def\se {\subseteq}
\def\sm{\setminus}

\def\al{\alpha}

\def\eq{\equiv}

\def\da{\delta}

\def\Proof{\noindent{\it Proof}}
\def\Ack{\medskip\noindent {\bf Acknowledgment}}

\theoremstyle{plain}
\newtheorem{theorem}{Theorem}

\newtheorem{lemma}{Lemma}

\newtheorem{conjecture}{Conjecture}
\theoremstyle{definition}
\newtheorem{definition}{Definition}

\theoremstyle{remark}
\newtheorem{remark}{Remark}

\makeatletter
\@namedef{subjclassname@2020}{%
  \textup{2020} Mathematics Subject Classification}
\makeatother

\allowdisplaybreaks

 \vspace{4mm}

\begin{document}
\hbox{Preprint, {\tt arXiv:2606.18234}}
\medskip

\title
[On zero-sum problems of new types]
{On zero-sum problems of new types}

\author
[Zhi-Wei Sun] {Zhi-Wei Sun}

\address{School of Mathematics, Nanjing
University, Nanjing 210093, People's Republic of China}
\email{zwsun@nju.edu.cn}

\keywords{Zero-sum, EGZ theorem, congruence.
\newline \indent Supported by the Natural Science Foundation of China (grant 12371004).}

\begin{abstract} In this paper, we investigate zero-sum problems of new types. For example, given $2n-1$ integers $a_1,\ldots,a_{2n-1}$ not divisible by an integer $n>1$, we prove that for some nonempty $I\subseteq\{1,\ldots,2n-1\}$ with $|I|\leqslant n$, the sum $\sum_{i\in I}a_i$ is divisible by $n$ but not divisible by $n^2$. We also pose several conjectures for further research.
\end{abstract}
\maketitle

\section{Introduction}
\setcounter{lemma}{0}
\setcounter{theorem}{0}
\setcounter{corollary}{0}
\setcounter{remark}{0}
\setcounter{definition}{0}
\setcounter{equation}{0}

Let $\N=\{0,1,2,\ldots\}$ and $\Z^+=\{1,2,3,\ldots\}$.
In 1961 P. Erd\H os, A. Ginzburg and A. Ziv \cite{EGZ} proved the following classical
theorem.

\medskip
\noindent {\bf EGZ Theorem.} {\it Let $a_1,\ldots,a_{2n-1}\in\Z$ with $n\in\Z^+$.
Then we have $\sum_{i\in I}a_i\eq0\pmod n$ for some $I\se\{1,\ldots,2n-1\}$ with $|I|=n$.}
\medskip

For a finite additive abelian group $G$, its {\it EGZ constant} $s(G)$
is defined as the least positive integer $k$ such that any sequence over $G$ of length $k$
has a zero-sum subsequence of length $\exp(G)$, where $\exp(G)$ is the exponent of $G$.
The EGZ theorem is equivalent to $s(\Z_n)=2n-1$, where $\Z_n$ denotes the cyclic group $\Z/n\Z$.

Let $n\in\Z^+$, and consider $n-1$ copies of the four ordered pairs
$$(0,0),\ (0,1),\ (1,0),\ (1,1).$$
Clearly, no $n$ of them sum to an ordered pair congruent to $(0,0)$ modulo $n$.
So, $s(\Z_n^2)>4n-4$. On the other hand,
a conjecture of A. Kemnitz \cite{K} confirmed by C. Reiher \cite{R}
states that any sequence over $\Z_n^2$ of length at least $4n-3$
contains a zero-sum subsequence of length $n$. So $s(\Z_n^2)=4n-3$.

 For any prime $p$,
the $p$-adic valuation of an integer $m$ is defined as
$$\nu_p(m)=\sup\{n\in\N:\ p^n\mid m\},$$
and we also write $p^n\|m$ when $\nu_p(m)=n\in\N$.
At the end of \cite{G25}, W. Gao et al. introduced the following definition involving $p$-adic valuations.

\begin{definition} [Gao-Hui-Jiang-Li-Wang] Let $p$ be a prime and let $\al\in\Z^+$. Define the constant $s(p,\al)$ as
the smallest positive integer $k$ such that for any integers
$a_1,\ldots,a_k\not\eq0\pmod p$ we can select $p$ of them for which the $p$-adic valuation of their sum is exactly $\al$.
\end{definition}

It is easy to see that $s(2,1)=3$.
For any odd prime $p$, Gao et al. \cite{G25} showed that $s(p,1)>2p$ and conjectured that $s(p,1)=2p+1$. In a recent paper, Gao et al. \cite{G26} proved that $s(p,1)\ls 3p-2$ for any prime $p>3$.

For a finite additive abelian group $G$, its {\it Davenport constant} $D(G)$ is defined as the least positive integer $k$ such that for any $a_1,\ldots,a_k\in G$ there is a nonempty
$I\se\{1,\ldots,k\}$ such that $\sum_{i\in I}a_i=0$. If $a_1,\ldots,a_k\in G$ with $k\gs|G|$, then
$$0,\ a_1,\ a_1+a_2,\ \ldots,\ a_1+\cdots +a_k$$
all belong to $G$ and they cannot be pairwise different by the Pigeonhole Principle.
 Thus $D(G)\ls |G|$.  It is known that $D(\Z_n)=n$ and $D(\Z_n^2)=2n-1$ for any $n\in\Z^+$ (cf. J. E. Olson \cite{O}).

With the above background, in this paper we study zero-sum problems of new types.

\begin{definition}\label{Def-w} Let $n>1$ and $r>0$ be integers. Define the new constant $w_r(n)$ as the least positive integer $k$ such that for any $k$ integer vectors ${\bf a}_i=(a_{i1},\ldots,a_{ir})\not\eq{\bf 0}=(0,\ldots,0)\pmod n$ $(i=1,\ldots,k)$, there is an $I\se\{1,\ldots,k\}$ such that
$\sum_{i\in I}{\bf a}_i\equiv{\bf 0}\pmod n$ but $\sum_{i\in I}{\bf a}_i\not\equiv{\bf 0}\pmod {n^2}.$
\end{definition}

\begin{remark} For a general finite abelian group
$G\cong\Z_{n_1}\oplus\cdots\oplus \Z_{n_r}$ with $1<n_1\mid n_2\mid\cdots\mid n_r$, we may define
the constant $w(G)$ as the least positive integer $k$ such that among any $k$ integer vectors
${\bf a}_1,\ldots,{\bf a}_k\in\Z^r$ not congruent to ${\bf 0}=(0,\ldots,0)\in\Z^r$ modulo the vector ${\bf n}=(n_1,\ldots,n_r)$ we can select some of them such that their sum is congruent to ${\bf 0}$
modulo ${\bf n}$ but not congruent to ${\bf 0}$ modulo ${\bf n}^2=(n_1^2,\ldots,n_r^2)$.
With this notation, we have $w_r(n)=w(\Z_n^r)$ for any integer $n>1$.
\end{remark}

Now we state our first and second theorems.

\begin{theorem}\label{Th-2n-1} Let $n>1$ be an integer. Then $w_1(n)=2n-1$.
Moreover, for any integers $a_1,\ldots,a_{2n-1}$
 not divisible by $n$,  there is a nonempty $I\se\{1,\ldots,2n-1\}$
with $|I|\ls n$ such that $\sum_{i\in I}a_i$ is divisible by $n$ but not divisible by $n^2$.
\end{theorem}

\begin{theorem} \label{Th-2n-2} Let $n>1$ be an integer, and let
$a_1,\ldots,a_{2n-2}$ be integers relatively prime to $n$. Then, there is no nonempty $I\se\{1,\ldots,2n-1\}$ such that $\sum_{i\in I}a_i$
is divisible by $n$ but not divisible by $n^2$, if and only if
\begin{equation}\label{b}
\begin{aligned}&\ |\{1\ls i\ls 2n-2:\ a_i\eq b\pmod{n^2}\}|
\\=&\ |\{1\ls i\ls 2n-2:\ a_i\eq -b\pmod{n^2}\}|=n-1.
\end{aligned}
\end{equation}
for some integer $b$ relatively prime to $n$.
\end{theorem}

Motivated by Theorems \ref{Th-2n-1}-\ref{Th-2n-2} and the fact that $s(\Z_n^2)=4n-3$, we pose the following conjecture
based on our computation.

\begin{conjecture} Let $n>1$ be an integer.

{\rm (i)} For any $2n-1$ integers $a_1,\ldots,a_{2n-1}$ not divisible by $n$, there is an $I\se\{1,\ldots,2n-1\}$ with $|I|\in\{n,n+1\}$ such that $\sum_{i\in I}a_i$ is divisible by $n$
but not divisible by $n^2$.

{\rm (ii)} We have $w_2(n)=4n-3$.
\end{conjecture}
\begin{remark} It is interesting to study whether $w_r(n)=s(\Z_n^r)$ for any integers $n>1$ and $r>0$.
\end{remark}

\begin{definition} \label{Def1.2} Let $n>1$ and $r>0$ be integers.

{\rm (i)} Define $s_r(n)$ as the least positive integer $k$ such that for any $k$ ordered integer tuples
$${\bf a}_1=(a_{11},\ldots,a_{1r}),\ \ldots,\ {\bf a}_k=(a_{k1},\ldots,a_{kr})$$
not congruent to ${\bf 0}=(0,\ldots,0)$ modulo $n$, there is an $I\se\{1,\ldots,k\}$ with $|I|=n$
for which $$\sum_{i\in I}{\bf a}_i\eq{\bf 0}\pmod n
\ \ \t{but}\ \sum_{i\in I}{\bf a}_i\not\eq{\bf 0}\pmod {n^2}.$$

{\rm (ii)} Define $t_r(n)$ as the least positive integer $k$ such that for any $k$ ordered integer tuples
$${\bf a}_1=(a_{11},\ldots,a_{1r}),\ \ldots,\ {\bf a}_k=(a_{k1},\ldots,a_{kr})$$
with all the $a_{ij}\ (1\ls i\ls k,\ 1\ls j\ls r)$ relatively prime to $n$, there is an $I\se\{1,\ldots,k\}$ with $|I|=n$ for which $$\sum_{i\in I}{\bf a}_i\eq{\bf 0}\pmod n
\ \ \t{but}\ \sum_{i\in I}{\bf a}_i\not\eq{\bf 0}\pmod {n^2}.$$
\end{definition}

\begin{remark} In the spirit of Remark 1.1, we can extend our definitions of $s_r(n)$ and $t_r(n)$
to any finite abelian group $G$ with $|G|>1$.
\end{remark}

By Definition \ref{Def1.2}, we obviously have $s_r(n)\gs t_r(n)$ for any integers $n>1$ and $r>0$.
Note also that $s_1(p)=t_1(p)=s(p,1)$ for any prime $p$.
\medskip

Our third theorem provides lower and upper bounds for $s_r(n)$ and $t_r(n)$.

\begin{theorem}\label{Th-bound} Let $n>1$ and $r>0$ be integers.

{\rm (i)} We have
\begin{equation}\label{st<=} s_r(n)\ls n^{r+1}-n+1\ \ \t{and}\ \ t_r(n)\ls n\varphi(n)^r+1,
\end{equation}
where $\varphi$ is Euler's totien function.

{\rm (ii)} If $n>2$ then $s_r(n)\gs 2nr+1$. If $n$ is odd, then
\begin{equation}\label{t12} t_1(n)\gs 2n+1\ \ \t{and}\ \  t_2(n)\gs4n+1.
\end{equation}

{\rm (iii)} If $n\gs4$, then
\begin{equation}t_r(n)\ls s_r(n)\ls s(\Z_n^{2r}).
\end{equation}
When $n$ is even, we have
\begin{equation}\label{t<=}t_r(n)\ls n+s(\Z_{n/2}^r).
\end{equation}
\end{theorem}
\begin{remark} For any integer $n\gs4$, as $s(\Z_n^2)=4n-3$ we have $s_1(n)\ls 4n-3$,
which was pointed out by Gao et al. \cite{G26} when $n$ is prime.
For any $n,d\in\Z^+$, H. Harborth \cite{H} proved that $s(\Z_n^d)\ls (n-1)n^d+1$,
and the upper bound was improved to $c_d n$ by N. Alon and M. Dubiner \cite{AD}, where $c_d$
is a positive constant only depending on $d$.
\end{remark}

Our fourth theorem gives some exact values of $s_r(n)$ and $t_r(n)$.

\begin{theorem}\label{Th-value} {\rm (i)} For any $r\in\Z^+$, we have
\begin{equation} s_r(2)=2^{r+1}-1,\ t_r(2)=3,\ t_2(3)=3\times 2^r+1.
\end{equation}
Also, for any integer $r>1$ we have
\begin{equation}3\times2^r+1\ls s_r(3)\ls 2\times 3^r-1.
\end{equation}

{\rm (ii)} Let $n$ be a positive even integer. Then
\begin{equation}t_1(n)=2n-1\ \ \t{and}\ \ t_2(n)=3n-3.
\end{equation}
\end{theorem}
\begin{remark} It is easy to see that $s_r(3)=2\times3^r+1$ for $r=1,2$ but
$s_3(3)>3\times2^3+1=25$.
\end{remark}

Based on Theorem \ref{Th-bound} and our computation, we pose the following conjecture.

\begin{conjecture}\label{Conj1.2} Let $n>2$ be an integer. Then
\begin{equation}\label{st}
s_1(n)=2n+1\ \ \t{and}\ \ s_2(n)=4n+1.
\end{equation}
\end{conjecture}
\begin{remark} For any odd integer $n>2$,  \eqref{st} and \eqref{t12}
together imply that $t_1(n)=2n+1$ and $t_2(n)=4n+1$.
\end{remark}

Now we state our fifth theorem.

\begin{theorem} \label{Th1.2} Let $n\gs2$ be an integer and let $a_1,\ldots,a_{2n+\da_n}$ be integers not divisible by $n$, where $\da_n\in\{\pm1\}$ is given by
$$\da_n=\begin{cases}-1&\t{if}\ 2\mid n\ \t{and}\ a_1\eq\cdots\eq a_{2n-1}\pmod 2,
\\1&\t{otherwise}.\end{cases}$$
Suppose that
 $$|\{1\ls i\ls 2n+\da_n:\ a_i\eq r\pmod n\}|\gs n$$
 for some $r\in\{1,\ldots,n-1\}$.
 Then, for some $I\se\{1,\ldots,2n+\da_n\}$ with $|I|=n$, the sum $\sum_{i\in I}a_i$
 is divisible by $n$ but not divisible by $n^2$.
 \end{theorem}
 \begin{remark} When $n$ is an odd prime,
 Theorem \ref{Th1.2} was essentially obtained in the proof of
  \cite[Lemma 3.4]{G25} via the Cauchy-Davenport theorem
  (cf. M. B. Nathanason \cite[p.\,44]{N96}). Our proof of Theorem \ref{Th1.2}
  uses a different approach which also works when $n$ is composite.
 \end{remark}

Our last theorem is as follows.

\begin{theorem}\label{Th1.3} Let $p$ be an odd prime, and let $a_1,\ldots,a_{2p+1}$
be integers not divisible by $p$ with
$$\max_{1\ls r\ls p-1}|\{1\ls i\ls 2p+1:\ a_i\eq r\pmod p\}|\ls p-1.$$
Suppose that $p\|(a_s+a_t-a_u-a_v)$ for some $s,t,u,v\in\{1,\ldots,2p+1\}$ with $s\not=t$ and $u\not=v$. Then, for some $I\se\{1,\ldots,2p+1\}$ with $|I|=p$ we have $p\|\sum_{i\in I}a_i$,
unless
$$\max_{1\ls r\ls p-1}|\{1\ls i\ls 2p+1:\ i\not=s,t,u,v\ \t{and}\ a_i\eq r\pmod p\}|=p-1$$
and
$$|\{a_i+p\Z:\ i\in\{1,\ldots,2p+1\}\sm\{s,t,u,v\}\}|\gs3.$$
\end{theorem}

We are going to prove Theorems \ref{Th-2n-1}-\ref{Th-2n-2}, Theorems \ref{Th-bound}-\ref{Th-value} and
Theorems \ref{Th1.2}-\ref{Th1.3} in Sections 2, 3 and 4, respectively.

\section{Proofs of Theorems \ref{Th-2n-1} and \ref{Th-2n-2}}
\setcounter{lemma}{0}
\setcounter{theorem}{0}
\setcounter{corollary}{0}
\setcounter{remark}{0}
\setcounter{equation}{0}

\begin{lemma} \label{Lem-H} Let $n>1$ be an integer and let $x_1,\ldots,x_{2n-1}$
be nonzero element of the cyclic group $\Z_n=\Z/n\Z$. Then, either
$$\l\{\sum_{i\in I}x_i:\ I\se J\r\}=\Z_n$$
for some $J\se\{1,\ldots,2n-1\}$ with $|J|=n-1$, or there is a nontrivial proper subgroup
$H$ of $\Z_n$ with $|\{1\ls i\ls 2n-1:\ x_i\in H\}|\gs 2|H|-1$.
\end{lemma}
\Proof. Suppose that there is no nontrivial proper subgroup
$H$ of $\Z_n$ with $|\{1\ls i\ls 2n-1:\ x_i\in H\}|\gs 2|H|-1$.
We aim to find distinct numbers $j_1,\ldots,j_{n-1}$
among $1,\ldots,2n-1$ greedily so that the set $J=\{j_1,\ldots,j_{n-1}\}$ meets the purpose.

Let $j_1=1$. Then $A_1=\{0,x_{j_1}\}$ has cardinality two.
Suppose that we have chosen  $k<n-1$ distinct elements $j_1,\ldots,j_k$ from $\{1,\ldots,2n-1\}$
such that
$$A_k=\l\{\sum_{i\in I}x_{i}:\ I\se\{j_1,\ldots,j_k\}\r\}$$
has cardinality at least $k+1$.
We want to find $j_{k+1}\in\{1,\ldots,2n-1\}\sm \{j_1,\ldots,j_k\}$
such that
$$A_{k+1}=\l\{\sum_{i\in I}x_{i}:\ I\se\{j_1,\ldots,j_k,j_{k+1}\}\r\}$$
has cardinality at least $k+2$.
This is easy if $A_k=\Z_n$.

Now assume that
$A_k\not=\Z_n$ and set $J_k=\{j_1,\ldots,j_k\}$.
If $x_j+A_k\not=A_k$ for some $j\in\{1,\ldots,2n-1\}\sm J$, then we take such a $j$
as $j_{k+1}$ and note that $A_k\subset A_{k+1}$ and hence $|A_{k+1}|\gs |A_k|+1\gs k+2$.

Assume that $x_j+A_k=A_k$ for all $j\in\{1,\ldots,2n-1\}\sm J_k$.
Then $\{x_j:\ 1\ls j\ls 2n-1\ \&\ j\not\in J_k\}$ is a subset of the stabilizer
$H=\{h\in\Z_n:\ h+A_k=A_k\}$ of $A_k$. As $|J_k|<2n-1$ and $x_j\not=0$ for all $1\ls j\ls2n-1$,
we see that $H\not=\{0\}$. Note also that $H\se A_k\subset \Z_n$.
So $H$ is a nontrivial proper subgroup of $\Z_n$. As $A_k=A_k+H$
is a proper union of cosets of $H$, we have $k+1\ls|A_k|\ls n-|H|$.
Note that
\begin{align*}|\{1\ls j\ls 2n-1:\ j\not\in J_k\}|&=2n-1-k
\\&\gs 2n-1-(n-|H|-1)=n+|H|
\\&\gs 2|H|-1.
\end{align*}
This contradicts our supposition that such a subgroup $H$ does not exist.

In view of the above, we can find $J=\{j_1,\ldots,j_{n-1}\}\se\{1,\ldots,2n-1\}$ with $|J|=n-1$
such that $A_{n-1}=\{\sum_{i\in I}x_i:\ I\se J\}$ has cardinality $(n-1)+1$ and hence $A_{n-1}=\Z_n$.
This concludes our proof of Lemma \ref{Lem-H}. \qed

\medskip
\noindent {\it Proof of Theorem \ref{Th-2n-1}}. Among the $2n-2$ integers consisting of $n-1$ copies of $1$ and $n-1$ copies of $-1$, if the sum of some of them is divisible by $n$ then the sum is zero and hence divisible by $n^2$. So we have $w_1(n)>2n-2$.

Below we use induction on $n$ to show the last assertion in Theorem \ref{Th-2n-1}
which implies the inequality $w_1(n)\ls 2n-1$. The case $n=2$ is trivial.

Now let $n>2$ and assume that the desired result holds if $n$ becomes smaller.

Suppose that the desired result fails. We claim that for any nonempty $I\se\{1,\ldots,2n-1\}$  we have
$$S(I)\eq0\pmod n \Rightarrow S(I)\eq0\pmod{n^2},$$
where $S(I)=\sum_{i\in I}a_i$. In fact, when $n\mid S(I)$ for some nonempty $I\se\{1,\ldots,2n-1\}$,
as $D(\Z_n)=n$ we can partition $I$ as $I_1\cup \cdots\cup I_t$ with $0<|I_s|\ls n$ for all $s=1,\ldots,t$
such that $n\mid S(I_s)$ for all $s=1,\ldots,t$.
By the supposition, $S(I_s)\eq 0\pmod{n^2}$ for all $s=1,\ldots,t$, and hence
$S(I)=\sum_{s=1}^tS(I_s)\eq0\pmod{n^2}$. So the claim holds.

{\it Case}\ 1. There is a nontrivial proper subgroup $H$ of $\Z_n$ with $|\{1\ls i\ls 2n-1:\
a_i+n\Z\in H\}|\gs 2|H|-1$.

In this case, we have $H=d\Z$ with $d=n/|H|$. If $a_i+n\Z\in H$, then $d\mid a_i$.
As $2\ls |H|<n$, by the induction hypothesis, for some nonempty $I\se\{1\ls i\ls 2n-1:\ d\mid a_i\}$
with $|I|\ls |H|$ we have
$|H|\mid \sum_{i\in I}\f{a_i}d$ and $|H|^2\nmid \sum_{i\in I}\f{a_i}d$.
Thus $n=d|H|$ divides $S(I)$ but $n^2=d^2|H|^2$ does not divide $S(I)$.
This contradicts our supposition.

{\it Case} 2. There is no nontrivial proper subgroup $H$ of $\Z_n$ with $|\{1\ls i\ls 2n-1:\
a_i+n\Z\in H\}|\gs 2|H|-1$.

In this case, by Lemma \ref{Lem-H}, for some $J\se\{1,\ldots,2n-1\}$ with $|J|=n-1$, we have
$\{S(I)+n\Z:\ I\se\ J\}=\Z_n$.
For any nonempty $I_0\se\{1,\ldots,2n-1\}\sm J$, there is a $J_0\se J$ such that
$S(I_0)\eq-S(J_0)\pmod n$ and hence
$$S(I_0)\eq-S(J_0)\pmod {n^2},$$
which shows that $S(I_0)$ modulo $n^2$ only depends on $S(I_0)$
modulo $n$.

Write $\bar J=\{1,\ldots,2n-1\}\sm J=\{i_1,\ldots,i_n\}$. By the last paragraph, for the set
$${\mathcal S}=\{S(I)+n^2\Z:\ \emptyset \not=I\se \bar J\},$$
we have
\begin{equation}\label{S(I)}
|{\mathcal S}|=|\{S(I)+n\Z:\ \emptyset \not=I\se\bar J\}|\ls n.
\end{equation}
For $k=1,\ldots,n$, define
$${\mathcal S}_k=\{S(I)+n^2\Z:\ I\se\{i_1,\ldots,i_k\}\}.$$
Then ${\mathcal S}_1=\{n^2\Z,a_{i_1}+n^2\Z\}$ has cardinality two.
Note that ${\mathcal S_{k}}={\mathcal S}_{k-1}+\{n^2\Z,a_{i_k}+n^2\Z\}$
for all $k=2,\ldots,n$.

 If ${\mathcal S}_k={\mathcal S}_{k-1}$ for some $2\ls k\ls n$,
then
$$ H=\{h+n^2\Z:\ {\mathcal S}+(h+n^2\Z)={\mathcal S}\}$$
is a subgroup of $\Z_{n^2}$ containing $a_{i_k}+n^2\Z$.
As $n\nmid a_{i_k}$, we have $|H|>1$. Let $p$ be the smallest prime divisor of $|H|$.
Then $p\mid |H|$ and hence $p\mid n^2$. Thus $p\mid n$. The only subgroup of $H$
of order $p$ is $\{(n^2/p)q+n^2\Z:\ q=1,\ldots,p\}$. Thus $n^2/p$ divides $a_{i_k}$
which contradicts $n\nmid a_{i_k}$.

By the above, $|{\mathcal S}_k|\gs |\mathcal S_{k-1}|+1$ for all $k=2,\ldots,n$.
Therefore
$$|{\mathcal S}\cup\{n^2\Z\}|=|{\mathcal S}_n|\gs|S_1|+n-1\gs n+1.$$
Combining this with \eqref{S(I)}, we obtain that $n^2\Z\not\in {\mathcal S}$.

As $D(\Z_n)=n$, for some nonempty $I\se \bar J=\{i_1,\ldots,i_{n}\}$, we have $S(I)=\sum_{i\in I}a_i\eq0\pmod n$
and thus $S(I)\eq0\pmod{n^2}$, which contradicts with $n^2\Z\not\in {\mathcal S}$.

In view of the above, we have completed our proof of Theorem \ref{Th-2n-1}. \qed

The following lemma is motivated by Chowla's extension (cf. \cite[pp.\,43-44]{N96}) of the Cauchy-Davenport theorem
or the original proof of the EGZ theorem (cf. \cite{EGZ}).

\begin{lemma} \label{Lem-n} Let $n>1$ be an integer, and let $a_1,\ldots,a_{m}$ (with $m<n$)
be integers relatively prime to $n$. Then we have
$$\l|\l\{\sum_{i\in I}a_i+n\Z:\ I\se\{1,\ldots,m\}\r\}\r|\gs m+1.$$
\end{lemma}
\Proof. For $k=1,\ldots,m$, we define
$$S_k=\l\{\sum_{i\in I}a_i+n\Z:\ I\se\{1,\ldots,k\}\r\}.$$

Clearly, $S_1=\{n\Z,a_1+n\Z\}$ has cardinality two.

Now let $1<k\ls m$ and assume that $|S_{k-1}|\gs k$. Clearly $S_k=S_{k-1}+\{n\Z,a_k+n\Z\}$.
If $S_{k-1}\not=S_k$, then $|S_k|=|S_{k-1}|+1\gs k+1$.
If $S_{k-1}=S_k$, then
$$S_{k-1}+(a_k+n\Z)=S_{k-1}, \ S_{k-1}+2a_k+n\Z=S_{k-1},\ \ldots,\ S_{k-1}+na_k+n\Z=S_{k-1},$$
and hence $|S_k|=|S_{k-1}|\gs n\gs k+1$ since $\{a_k,2a_k,\ldots,na_k\}$
is a complete system of residues modulo $n$.

By the above, we have proved that $|S_k|\gs k+1$ for all $k=1,\ldots,m$ by induction.
So $|S_{m}|\gs m+1$ as desired. \qed

\medskip
\noindent{\it Proof of Theorem \ref{Th-2n-2}}.
We first prove the `if' direction. Suppose that \eqref{b} holds for some $b\in\Z$
with $\gcd(b,n)=1$. If $lb+m(-b)\eq0\pmod n$ with $l,m\in\{0,\ldots,n-1\}$ and $l+m>0$, then
$l=m$ and $lb+m(-b)=0$, so there is no nonempty $I\se\{1,\ldots,2n-1\}$ with $\sum_{\in I}a_i$
divisible by $n$ but not divisible by $n^2$.

 Below we prove the `only if' direction. Assume that there is no nonempty $I\se\{1,\ldots,2n-1\}$ such that $\sum_{i\in I}a_i$
is divisible by $n$ but not divisible by $n^2$.
Set $S(I)=\sum_{i\in I}a_i$ for all $I\se\{1,\ldots,2n-2\}$.
By Lemma \ref{Lem-n},  the set $\{S(J)+n\Z:\ J\se\{n,\ldots,2n-2\}\}$
has cardinality at least $(n-1)+1$ and hence this set coincides with $\Z_n$.

Let $\emptyset\not=I\se\{1,\ldots,n-1\}$. Then, for some $J\se\{n,\ldots,2n-2\}$ we have
$S(J)\eq-S(I)\pmod n$ and hence $S(I\cup J)\eq0\pmod n$. By the assumption, we must have
$S(I\cup J)\eq0\pmod{n^2}$ and hence $S(I)\eq-S(J)\pmod{n^2}$.
If $\emptyset\not=I'\se\{1,\ldots,n-1\}$ and $S(I')\eq S(I)\pmod n$, then we also have
$S(I')\eq- S(J)\pmod{n^2}$ and hence $S(I')\eq S(I)\pmod {n^2}$.

By the last paragraph, for the set
$${\mathcal S}=\{S(I)+n^2\Z:\ \emptyset \not=I\se\{1,\ldots,n-1\}\},$$
we have
\begin{equation}\label{<=}
|{\mathcal S}|=|\{S(I)+n\Z:\ \emptyset \not=I\se\{1,\ldots,n-1\}\}|\ls n.
\end{equation}
As $a_1,\ldots,a_n$ are relatively prime to $n^2$, by Lemma \ref{Lem-n} we have
\begin{equation}\label{>=} |{\mathcal S}\cup\{n^2\Z\}|=|\{S(I)+n^2\Z:\  I\se\{1,\ldots,n-1\}\}|\gs n.
\end{equation}
Combining \eqref{<=} with \eqref{>=}, we get that
$$|{\mathcal S}|\not=n\Rightarrow |{\mathcal S}|=n-1\ \t{and}\ n^2\Z\not\in {\mathcal S}.$$
If $|{\mathcal S}|=n$, then
$$\{S(I)+n\Z:\ \emptyset \not=I\se\{1,\ldots,n-1\}\}=\Z_n,$$
hence for certain nonempty $I\se\{1,\ldots,n-1\}$ we have $S(I)\eq0\pmod n$ and thus $S(I)\in n^2\Z$.
So, we always have $|{\mathcal S}\cup\{n^2\Z\}|=n$.

For $k=1,\ldots,n-1$ define
$${\mathcal S}_k=\{S(I)+n^2\Z:\ I\se\{1,\ldots,k\}\}.$$
Then ${\mathcal S}_1=\{n^2\Z,a_1+n^2\Z\}$ has cardinality two, and ${\mathcal S}_n={\mathcal S}\cup\{n^2\Z\}$ has cardinality $n$.
Note that ${\mathcal S_{k}}={\mathcal S}_{k-1}+\{n^2\Z,a_k+n^2\Z\}$
for all $k=2,\ldots,n-1$. If ${\mathcal S}_k={\mathcal S}_{k-1}$ for some $2\ls k\ls n-1$, then
\begin{align*}&\ {\mathcal S}_{k-1}+(a_k+n^2\Z)={\mathcal S}_{k-1},
\\&\
{\mathcal S}_{k-1}+(2a_k+n^2\Z)={\mathcal S}_{k-1},
\\& \ldots,\
{\mathcal S}_{k-1}+(n^2a_k+n^2\Z)={\mathcal S}_{k-1},
\end{align*}
and hence $|{\mathcal S}_{k-1}|\gs n^2>n+1$ which contradicts
the fact that $|{\mathcal S}_{k-1}|\ls |{\mathcal S}_{n-1}|=n$.
Thus ${\mathcal S}_{k-1}\subset {\mathcal S}_k$ for all $k=2,\ldots,n-1$.
Since
$$\sum_{1<k<n} |{\mathcal S}_k\sm {\mathcal S}_{k-1}|=
|{\mathcal S}_n|-|{\mathcal S}_1|=n+1-2=n-1,$$
we have $|{\mathcal S}_k\sm {\mathcal S}_{k-1}|=1$ for all $1<k<n$.
Therefore $|{\mathcal S}_k|=k+1$ for each $k=1,\ldots,n-1$.

We claim that for each $k=1,\ldots,n-1$, we have
$$a_k+n^2\Z\in\{a_1+n^2\Z,-a_1+n^2\Z\}$$
and
$$ {\mathcal S}_k=\{ja_1+n^2\Z:\ j=j_k,\ldots,j_k+k\}\ \ \t{for some}\ j_k\in\Z.$$
This holds trivially for $k=1$.

Now, let $1<k<n$ and suppose that
$$ {\mathcal S}_{k-1}=\{ja_1+n^2\Z:\ j=j_{k-1},\ldots,j_{k-1}+k-1\}$$
with $j_{k-1}\in\Z.$ As $\gcd(a_1a_k,n)=1$, we have $a_1q\eq a_k\pmod{n^2}$ for some $1\ls q<n^2$
with $\gcd(q,n)=1$. Let $J=\{j_{k-1},\ldots,j_{k-1}+k-1\}$. Then
$${\mathcal S}_k=\{ja_1+n^2\Z:\ j\in J\cup (q+J)\}.$$
Note that $|{\mathcal S}_k|=k+1=|J|+1$.
If $2\ls q\ls n^2-k$, then
$$\max J<q+\max J-1<q+\max J\ls n^2-k+\max J=\min J+n^2-1$$
and hence $|{\mathcal S}_k|\gs|J|+2$ which leads to a contradiction.

Suppose that  $n^2-k<q\ls n^2-2$. Then  $q>n^2-n\gs n>k-1$, $q'=n^2-q\in\{2,\ldots,k-1\}$, and
$$\{q+j+n^2\Z:\ j\in J\}=\{-q'+j+n^2\Z:\ j\in J\}.$$
Thus
\begin{align*}\min J-q'&\ls \min J+1-q'<\min J=j_{k-1}
\\&<\max J=\min J+k-1<\min J-q'+n^2
\end{align*}
and hence $|{\mathcal S}_k|\gs|J|+2$ which leads to a contradiction.

In view of the last two paragraphs, we must have $q\in\{1,n^2-1\}$, and hence $a_k$ is congruent to $a_1$ or $-a_1$ modulo $n^2$. If $a_k\eq a_1\pmod {n^2}$, then
$${\mathcal S}_k=\{j+n^2\Z:\ j=j_{k-1},\ldots,j_{k-1}+k\}.$$
If $a_k\eq -a_1\pmod {n^2}$, then
$${\mathcal S}_k=\{j+n^2\Z:\ j=j_{k-1}-1,j_{k-1},\ldots,j_{k-1}+k-1\}.$$

As we have proved the claim by induction, we have
$$\{a_k+n^2\Z:\ k=1,\ldots,n-1\}\se\{a_1+n^2\Z,-a_1+n^2\Z\}.$$
Similarly, for any $I\se\{1,\ldots,2n-2\}$ with $|I|=n-2$, we also have
$$\{a_k+n^2\Z:\ k\in\{1\}\cup I\}\se\{a_1+n^2\Z,-a_1+n^2\Z\}.$$
Therefore,
$$\{a_k+n^2\Z:\ k=1,\ldots,2n-2\}\se\{a_1+n^2\Z,-a_1+n^2\Z\}.$$
For
$$l=|\{1\ls k\ls 2n-2:\ a_i\eq a_1\pmod {n^2}$$
and $$m=|\{1\ls k\ls 2n-2:\ a_i\eq a_1\pmod {n^2},$$
we have $l+m=2(n-1)$. If $l\gs n$ then for some $I\se\{1,\ldots,2n-2\}$ with $|I|=n$
we have $a_i\eq a_1\pmod{n^2}$ for all $i\in I$, and hence $\sum_{\in I}a_i\eq na_1\not\eq0\pmod{n^2}$
which contradicts the assumption. Similarly, $m\gs n$ is also impossible. As $l+m=2(n-1)$, we must have $l=m=n-1$. So \eqref{b} holds for $b=a_1$.

In view of the above, we have completed our proof of Theorem \ref{Th-2n-2}. \qed

\section{Proofs of Theorems \ref{Th-bound}-\ref{Th-value}}
\setcounter{lemma}{0}
\setcounter{theorem}{0}
\setcounter{corollary}{0}
\setcounter{remark}{0}
\setcounter{equation}{0}

The following lemma is motivated by the proof of \cite[Lemma 3.4]{G25}.

\begin{lemma} \label{Lem2.1} Let $n>1$ and $r>0$ be integers. Let
${\bf a}_1,\ldots,{\bf a}_{n+1}$ be vectors in $\Z^r$ with ${\bf a}_1\eq\cdots\eq{\bf a}_{n+1}\not\eq {\bf 0}\pmod n$, where ${\bf 0}=(0,\ldots,0)$ is the zero vector in $\Z^r$.  Then, for some $I\se\{1,\ldots,n+1\}$ with $|I|=n$, we have
$\sum_{i\in I}{\bf a}_i\eq{\bf 0}\pmod n$ but $\sum_{i\in I}{\bf a}_i\not\eq{\bf 0}\pmod {n^2}$.
\end{lemma}
\Proof. Assume that the desired result fails. We want to deduce a contradiction.

Let $1\ls s<t\ls n+1$. Then the set $J=\{1,\ldots,n+1\}\sm\{s,t\}$ has cardinality $n-1$.
Since
$$\sum_{i\in J\cup\{s\}}{\bf a}_i\eq nr\eq\sum_{i\in J\cup\{t\}}{\bf a}_i\pmod n,$$
by the assumption we must have
$$\sum_{i\in J\cup\{s\}}{\bf a}_i\eq0\eq\sum_{i\in J\cup\{t\}}{\bf a}_i\pmod{n^2}$$
and hence ${\bf a}_s\eq {\bf a}_t\pmod{n^2}$.

By the last paragraph,  we have
$$\sum_{i=1}^n {\bf a}_i\eq \sum_{i=1}^n {\bf a}_1=n{\bf a}_1\pmod {n^2},$$
which yields a contradiction since ${\bf a}_1\not\eq0\pmod n$.
This concludes our proof. \qed

\medskip
\noindent{\it Proof of Theorem \ref{Th-bound}}. (i) There are totally $n^r-1$
nonzero vectors in $\Z_n^r$. By the Pigeonhole Principle, among any $k\gs n(n^r-1)+1$
vectors in $\Z^r$ not congruent to the zero vector ${\bf 0}=(0,\ldots,0)\in\Z^r$,
there are $n+1$ of them lying in the same residue class modulo $n$, and hence by Lemma
\ref{Lem2.1} we can select $n$ of them such that their sum is congruent to ${\bf 0}$
modulo $n$ but not congruent to ${\bf 0}$ modulo $n^2$. Thus $s_r(n)\ls n(n^r-1)+1$.
Similarly, we have $t_r(n)\ls n\varphi(n)^r+1$ since there are totally $\varphi(n)^r$ vectors in $(\Z_n^\times)^r$, where
$$\Z_n^\times=\{a+n\Z:\ 1\ls a\ls n\ \t{and}\ \gcd(a,n)=1\}).$$

(ii) Assume $n>2$. Set
$$a_1=\cdots =a_{n-1}=1,\ a_n=\cdots =a_{2n-2}=-1,\ a_{2n-1}=c,\ a_{2n}=-c,$$
where
$$c=\begin{cases}2&\t{if}\ 2\mid n,\\n-1&\t{if}\ 2\nmid n.
\end{cases}
$$

Let $I\se\{1,\ldots,2n\}$ with $|I|\ls n$,  and suppose that $n\mid \sum_{i\in I}a_i$
but $n^2\nmid \sum_{i\in I}a_i$.
As $|\sum_{i\in I}a_i|\ls (n-1)\times 1+c<2n$, we must have $\sum_{i\in I}a_i\in\{\pm n\}$.

If $\sum_{i\in I}a_i=n$ and $2\nmid n$, then
for some $0\ls s\ls n-2$ the multiset $\{a_i:\ i\in I\}$ consists of $c=n-1$, together with $s+1$ copies of $1$ and
$s$ copies of $-1$, hence $2\mid |I|$ and $|I|\not=n$.
Similarly, if $\sum_{i\in I}a_i=-n$ and $2\nmid n$ then we also have $2\mid |I|$ and $|I|\not=n$.

If $\sum_{i\in I}a_i=n$ and $2\mid n$, then
 the multiset $\{a_i:\ i\in I\}$ consists of $c=2$ and $n-2$ copies of $1$, hence $|I|=n-1\not=n$.
Similarly, if $\sum_{i\in I}a_i=-n$ and $2\mid n$ then we also have $|I|=n-1\not=n$.

By the above, $|I|$ is even if $n$ is odd, and $|I|=n-1$ if $n$ is even.

Now let's consider the following $2nr$ vectors in $\Z^r$:
\begin{align*}&(a_i,0,\ldots,0)\ (1\ls i\ls 2n),
\\& (0,a_i,0,\ldots,0)\ (1\ls i\ls 2n),
\\& \ldots,\ (0,\ldots,0,a_i)\ (1\ls i\ls 2n).
\end{align*}
By the above analysis, no $n$ of them sum to a vector in $\Z^r$
which is congruent to ${\bf 0}=(0,\ldots,0)\in\Z^r$ modulo $n$
but not so modulo $n^2$. Therefore $s_r(n)\gs 2nr+1$.

Now assume that $n$ is odd.
As $\pm1$ and $c=n-1$ are relatively prime to $n$,
by the above we also have $t_1(n)\gs 2n+1$.
Among the $4n$ ordered pairs
consisting of $(2,2),(2,-2),(-2,2),(-2,-2)$ and $n-1$ copies of
$(1,1),\ (1,-1),\ (-1,1),\ (-1,-1),$
obviously no $n$ of them sum to an ordered pair congruent to $(0,0)$ modulo $n$ but not so modulo $n^2$. Thus $t_2(n)\gs 4n+1$.

(iii) Now we turn to prove part (iii) of Theorem \ref{Th-bound}.
Consider  $k\gs s(\Z_n^{2r})$ integer vectors
$${\bf a}_1=(a_{11},\ldots,a_{1r}),\ \ldots,\ {\bf a}_k=(a_{k1},\ldots,a_{kr})$$
not congruent to ${\bf 0}=(0,\ldots,0)$ modulo $n$.
Write $a_{ij}=nq_{ij}+b_{ij}$ with $q_{ij},b_{ij}\in\Z$ and $0\ls b_{ij}\ls n-1$.
As $k\gs s(\Z_n^{2r})$, for the $k$ vectors
$$(q_{i1},\ldots,q_{ir},b_{i1},\ldots,b_{ir})\ (i=1,\ldots,k)$$
there is an $I\se\{1,\ldots,k\}$ with $|I|=n$
such that
$$\sum_{i\in I}q_{ij}\eq0\pmod n\ \t{and}\ \sum_{i\in I} b_{ij}\eq0\pmod n$$
for all $j=1,\ldots,r$. Note that
$$\sum_{i\in I}a_{ij}=n\sum_{i\in I}q_{ij}+\sum_{i\in I} b_{ij}\eq \sum_{i\in I}b_{ij}\pmod{n^2}.$$
If $\sum_{i\in I}{\bf a}_{i}\eq{\bf 0}\pmod {n^2}$, then, for each $j=1,\ldots,r$, we have
$$\sum_{i\in I}b_{ij}\eq0\pmod{n^2}$$
and hence $b_{ij}=0$ for all $i\in I$ since $0\ls \sum_{i\in I}b_{ij}\ls n(n-1)<n^2$.
As ${\bf a}_i\not\eq{\bf 0}\pmod n$ for all $i\in I$, we see that
$\sum_{i\in I}{\bf a}_{i}\not\eq{\bf 0}\pmod {n^2}$ although $\sum_{i\in I}{\bf a}_{i}\eq{\bf 0}\pmod {n}$. Therefore $s_r(n)\ls s(\Z_n^{2r})$.

Now assume that $n=2m$ for some $m\in\Z^+$. We
want to show that  $t_r(n)\ls n+s(\Z_m^r)$
Given $k\gs 2m+s(\Z_m^r)$ integer vectors ${\bf a}_i=(a_{i1},\ldots,a_{ir})$
 $(i=1,\ldots,k)$ with $a_{ij}$ relatively prime to $n$ for all $1\ls i\ls k$ and $1\ls j\ls r$,
 as $n$ is even we can write $a_{ij}=2q_{ij}+1$ with $q_{ij}\in\Z$. Since $k\gs s(\Z_m^r)$,
there is an $I_1\se\{1,\ldots,k\}$ with $|I_1|=m$ such that
$\sum_{i\in I_1}{\bf q}_i=m{\bf v}_1$
for some ${\bf v}_1\in\Z^r$, where ${\bf q}_i=(q_{i1},\ldots,q_{ir})$.
Similarly, for some $I_2\se\{1,\ldots,k\}\sm I$ with $|I_2|=m$, we have
$\sum_{i\in i_2}{\bf q}_i=m{\bf v}_2$
for some ${\bf v}_2\in\Z^r$.
Also, for some $I_3\se\{1,\ldots,k\}\sm(I_1\cup I_2)$ with $|I_3|=m$, we have
$\sum_{i\in i_3}{\bf q}_i=m{\bf v}_3$
for some ${\bf v}_3\in\Z^r$. Thus,
$$\sum_{i\in I_j}{\bf a_i}=\sum_{i\in I_j}(2{\bf q}_i+{\bf 1})=m{\bf u}_j$$
for all $j=1,2,3$, where ${\bf 1}=(1,\ldots,1)\in\Z^r$ and ${\bf u}_j=2{\bf v}_j+{\bf 1}$.
Note that
$$|I_1\cup I_2|=|I_1\cup I_3|=|I_2\cup I_3|=m+m=n$$
and
$$m{\bf u}_i+m{\bf u}_j\eq{\bf 0}\pmod {2m}\ \ \t{for all}\ 1\ls i< j\ls 3,$$
where ${\bf 0}$ is the zero vector in $\Z^r$.
By Lemma \ref{Lem2.1}, there are $1\ls i<j\ls 3$ such that
${\bf u}_i+{\bf u}_j\not\eq{\bf 0}\pmod 4$ and hence
${\bf u}_i+{\bf u}_j\not\eq{\bf 0}\pmod{4m}$.
Observe that
$$\sum_{h\in I_i\cup I_j}a_h\eq m({\bf u}_i+{\bf u}_j)\eq{\bf 0}\pmod{n}$$
but
$$\sum_{h\in I_i\cup I_j}a_h\eq m({\bf u}_i+{\bf u}_j)\not\eq{\bf 0}\pmod{n^2}.$$
So we do have the inequality $t_r(2)\ls n+s(\Z_{n/2}^r)$.

In view of the above, we have completed our proof of Theorem \ref{Th-bound}. \qed

\medskip
\noindent{\it Proof of Theorem \ref{Th-value}}. (i) In view of \eqref{st<=} with $n=2$, we have
$$s_r(2)\ls 2^{r+1}-2+1=2^{r+1}-1\ \ \t{and}\ \ t_r(2)\ls 2\varphi(2)^r+1=3.$$

Note that $t_r(2)>2$ since $(1,\ldots,1)+(-1,\ldots,-1)=(0,\ldots,0)$. So we have $t_r(2)=3$.

If ${\bf v}_1,\ldots,{\bf v}_{2^r-1}$ are all the nonzero vectors in $\{0,1\}^r$, then among the $2(2^r-1)=2^{r+1}-2$ vectors
$${\bf v}_1,\ -{\bf v}_1,\ \ldots,\ {\bf v}_{2^r-1},\ -{\bf v}_{2^r-1}$$
no two of them sum to a vector congruent to the zero vector ${\bf 0}\in\Z^r$ modulo $2$ but not congruent to ${\bf 0}$ modulo $4$. So $s_r(2)>2^{r+1}-2$ and hence $s_r(2)=2^{r+1}-1$.

 In view of \eqref{st<=} with $n=3$, we have
$$ t_r(3)\ls 3\varphi(3)^r+1=3\times2^r+1.$$
Consider the $3\times 2^r$ vectors consisting of
$-2{\bf v}$ and two copies of ${\bf v}$ for all ${\bf v}\in\{1,-1\}^r$.
If three of them sum to a vector congruent to the zero vector ${\bf 0}\in\Z^r$ modulo $3$, then
the three vectors are ${\bf v},{\bf v},-2{\bf v}$ for some ${\bf v}\in\{1,-1\}^r$, and their sum
is the zero vector. So $t_r(3)>3\times2^r$ and hence $t_r(3)=3\times2^r+1$.

Now assume $r>1$. Note that $s_r(3)\gs t_r(3)=3\times 2^r+1$. 
Below we show that $s_r(3)\ls 2\times 3^r-1$. 

 Let ${\bf a}_1,\ldots,{\bf a}_l$ be vectors in $\Z^r$
not congruent to the zero vector ${\bf 0}\in\Z^r$ modulo $3$, with
$l\gs 2\times 3^r-1$. 
Suppose that there is no $I\se\{1,\ldots,l\}$ with $|I|=3$ such that $\sum_{i\in I}{\bf a}_i
\eq{\bf 0}\pmod 3$ but $\sum_{i\in i}{\bf a}_i\not\eq{\bf 0}\pmod 9$. 
We want to deduce a contradiction. 

For each nonzero vector ${\bf x}\in V=\{0,1,-1\}^r$, set
$$m({\bf x})=|\{1\ls i\ls k:\ {\bf a}_i\eq {\bf x}\pmod3\}|.$$
Let ${\bf u},{\bf v},{\bf w}$ be three distinct nonzero elements of $V$
with ${\bf u}+{\bf v}+{\bf w}\eq{\bf 0}\pmod3$. We claim that $m({\bf u})+m({\bf v})+m({\bf w})\ls6$.

By Lemma \ref{Lem2.1} and the supposition, we have
$$\max\{m({\bf u}),m({\bf v}),m({\bf w})\}\ls 3.$$
If one of $m({\bf u}),m({\bf v}),m({\bf w})$ is zero, then
$m({\bf u})+m({\bf v})+m({\bf w})\ls0+3+3=6$. 

Now assume that $m({\bf u}),m({\bf v}),m({\bf w})\gs1$. 
Then there are $1\ls i<j<k\ls l$ such that
${\bf a}_i,{\bf a}_j,{\bf a}_k$ are congruent to the vectors ${\bf u},{\bf v},{\bf w}$, respectively. 
Then ${\bf a}_i+{\bf a}_j+{\bf a}_k\eq{\bf 0}\pmod 3$ and hence
${\bf a}_i+{\bf a}_j+{\bf a}_k\eq{\bf 0}\pmod 9$ by the supposition.
If one of $m({\bf u}),m({\bf v}),m({\bf w})$ is $3$, say, $m({\bf u})=3$, then there are distinct
$i_1,i_2\in\{1,\ldots,l\}\sm\{i,j,k\}$ such that 
$${\bf a}_{i_1}\eq {\bf a}_{i_2}\eq {\bf a}_i\eq {\bf u}\pmod3$$
and hence $${\bf a}_{i_1}\eq{\bf a}_{i_2}\eq {\bf a}_i\pmod 9$$
since ${\bf a}_{i_1}+{\bf a}_j+{\bf a}_k\eq{\bf 0}\pmod 9$ and ${\bf a}_{i_2}+{\bf a}_j+{\bf a}_k\eq{\bf 0}\pmod 9$, therefore ${\bf a}_{i_1}+{\bf a}_{i_2}+{\bf a}_i\eq3{\bf a}_i\not\eq{\bf 0}\pmod 9$ which leads to a contradiction.
Thus $m({\bf u}),m({\bf v}),m({\bf w})\ls2$ and hence $m({\bf u})+m({\bf v})+m({\bf w})\ls6$.
This proves the claim.

Let $S$ be the set of those $\{{\bf u},{\bf v},{\bf w}\}$ with ${\bf u},{\bf v},{\bf w}$
distinct elements of $V\sm\{{\bf 0}\}$ satisfying ${\bf u}+{\bf v}+{\bf w}\eq{\bf 0}\pmod3$.
Clearly, ${\bf v}\not\in\{{\bf 0},{\bf u},-{\bf u}\}$.
By the claim, we have
$$\sum_{\{{\bf u},{\bf v},{\bf w}\}\in S}(m({\bf u})+m({\bf v})+m({\bf w}))\ls 6|S|.$$
Observe that
$$(3!)|S|=\sum_{{\bf u}\in V\sm\{{\bf 0}\}}|\{{\bf v}\in V:\ {\bf v}\not=0,\pm{\bf u}\}|
=\sum_{{\bf u}\in V\sm\{{\bf 0}\}}(3^r-3)=(3^r-1)(3^r-3).$$
Also,
\begin{align*}&\ \sum_{\{{\bf u},{\bf v},{\bf w}\}\in S}(m({\bf u})+m({\bf v})+m({\bf w}))
\\=&\ \sum_{{\bf x}\in V\sm\{{\bf 0}\}}m({\bf x})|\{\{{\bf y},{\bf z}\}:\ {\bf y},{\bf z}\in V\sm\{0\}
\ \t{and}\ {\bf z}\eq -{\bf x}-{\bf y}\pmod3\}|
\\=&\ \sum_{{\bf x}\in V\sm\{{\bf 0}\}}m({\bf x})\f{|\{{\bf y}\in V:\ {\bf y}\not={\bf 0},\pm{\bf x}\}|}2=\f{3^r-3}2\sum_{{\bf x}\in V\sm\{{\bf 0}\}}m({\bf x})=\f{3^r-3}2l.
\end{align*}
Therefore,
$$\f{3^r-3}2l\ls 6|S|=(3^r-1)(3^r-3)$$
and hence $l\ls 2(3^r-1)$, which contradicts $l\gs 2\times3^r-1$.

 (ii) Recall the condition that $n$ is even.
 Among the $2(n-1)$ numbers consisting of $n-1$ copies of $1$
and $n-1$ copies of $-1$, no $n$ of them sum to an integer divisible by $n$ but not divisible by $n^2$.
Therefore $t_1(n)\gs 2n-1$. On the other hand,
$$t_1(n)\ls n+s(\Z_{n/2})=n+2\times\f n2-1=2n-1$$ by Theorem \ref{Th-bound}(iii).
Therefore $t_1(n)=2n-1$.

Now we turn to determine the value of $t_2(n)$. Let us consider the $3n-4$ ordered pairs
consisting of $n-1$ copies of $(1,-1)$ and $(-1,1)$, and $n/2-1$ copies of $(1,1)$ and $(-1,-1)$.
It is easy to see that no $n$ of them sum to an ordered pair congruent to $(0,0)$ modulo $n$ but not so modulo $n^2$ (since $s\times1+(n-s)(-1)=2s-n\not\eq0\pmod n$ for all $0<s<n/2$).
Therefore $t_2(n)\gs 3n-3$. On the other hand, by Theorem \ref{Th-bound}(iii) we have
$$t_2(n)\ls n+s(\Z_{n/2}^2)=n+4\times\f n2-3=3n-3.$$
Therefore $t_2(n)=3n-3$. 

In view of the above, we have completed our proof of Theorem \ref{Th-value}. \qed

\section{Proofs of Theorems \ref{Th1.2}-\ref{Th1.3}}
\setcounter{lemma}{0}
\setcounter{theorem}{0}
\setcounter{corollary}{0}
\setcounter{remark}{0}
\setcounter{equation}{0}

The following well-known lemma can be found in Theorem 5.1.10 of A. Geroldinger and F. Halter-Koch \cite[p.\,309]{GH}.

\begin{lemma} \label{Lem2.2} Let $n>1$ be an integer and let $a_1,\ldots,a_{n-1}\in\Z$. Then
$\sum_{i\in I}a_i\not\eq0\pmod n$ for all $\emptyset\not= I\se\{1,\ldots,n-1\}$ if and only if
$a_1\eq\cdots\eq a_{n-1}\eq c\pmod n$ for some integer $c$ relatively prime to $n$.
\end{lemma}

 \medskip
 \noindent{\it Proof of Theorem \ref{Th1.2}}.  By the EGZ theorem, the set
 $${\mathcal I}=\l\{I\se\{1,\ldots,2n+\da_n\}:\ |I|=n\ \t{and}\ \sum_{i\in I}a_i\eq0\pmod n\r\}$$
 is nonempty. Suppose that the desired result fails. Then, for any $I\in{\mathcal I}$ we have
 $\sum_{i\in I}a_i\eq0\pmod{n^2}$.

 Let $I\se\{1\ls i\ls 2n+\da_n:\ a_i\eq r\pmod n\}$ with $|I|=n$.
 As the desired result fails, by Lemma \ref{Lem2.1}
 we have $a_j\not\eq r\pmod n$ for some $j\in\{1,\ldots,m+n\}\sm I$.

 {\it Case} 1.  $a_i\not\eq a_j\pmod n$ for some $i,j\in\{1,\ldots,2n+\da_n\}\sm I$.

In this case, we may choose $J\se\{1,\ldots,2n+\da_n\}\sm I$
with $J=\{j_1,\ldots,j_{n-1}\}$, $j_1<\ldots<j_{n-1}$ and $|\{a_j+n\Z:\ j\in J\}|>1$. As $|\{a_j-r+n\Z:\ j\in J\}|>1$,
by Lemma \ref{Lem2.2} there is a nonempty subset $K$ of $J$ such that
 $\sum_{k\in K}(a_k-r)\eq0\pmod n$ and hence
$$\sum_{k\in K}a_k\eq r|K|\pmod n.$$

Let $s$ and $t$ be any two distinct elements of $I$. Choose $I_0\se I\sm\{s,t\}$ with $|I_0|=n-1-|K|$,
and set $I_s=I_0\cup \{s\}$ and $I_t=I_0\cup\{t\}$. Then
$$\sum_{i\in I_s\cup K}a_i\eq |I_s|r+\sum_{i\in K}a_i\eq(n-|K|)r+r|K|\eq0\pmod n$$
and hence $I_s\cup K\in\mathcal I$. Similarly, $I_t\cup K\in\mathcal I$. Thus
$$\sum_{i\in I_0\cup K\cup\{s\}}a_i\eq0\eq\sum_{i\in I_0\cup K\cup\{t\}}a_i\pmod{p^2}$$
and hence $a_s\eq a_t\pmod {p^2}$.

By the above, for some integer $r'\eq r\pmod n$, we have $a_i\eq r'\pmod{n^2}$ for all $i\in I$.
Thus
$$\sum_{i\in I}a_i\eq\sum_{i\in I}r'\eq nr'\not\eq0\pmod{n^2},$$
which leads to a contradiction since $r\not\eq0\pmod n$.

{\it Case} 2. For some $c\in\{1,\ldots,n-1\}\sm\{r\}$, we have $a_j\eq c\pmod n$ for all $j\in\{1,\ldots,2n+\da_n\}\sm I$.

In this case, we distinguish two subcases.

{\it Subcase} 2.1. $\da_n=1$.

In this subcase,  $J=\{1,\ldots,2n+\da_n\}\sm I$
has cardinality $n+1$. As $a_j\eq c\pmod n$ for all $j\in J$,
by Lemma \ref{Lem2.1} we get a contradiction.

{\it Subcase} 2.2. $\da_n=-1$.

In this subcase, $n$ is even, and all the numbers $a_1,\ldots,a_{2n-1}$ have the same parity.
As $r\eq c\pmod 2$ and $r\not=c$, $d=\gcd(r-c,n)$ is a proper divisor of $n$ with $2\mid d$.
For $k\in\{1,\ldots,n\}$, we clearly have
$$kr+(n-k)c\eq0\pmod n\iff n'\mid k\iff k\in\{\lambda n':\ \lambda=1,\ldots,d\},$$
where $n'=n/d$.

Let $s,t\in I$ with $s\not=t$. Then we can choose $I_1\se I\sm\{t\}$ with $|I_1|=n'$ and $s\in I_1$,
and $J_1\se J$
with $|J_1|=n-n'$. Note that
$$\sum_{i\in I_1\cup J_1}a_i\eq n'r+(n-n')c\eq n\f{r-c}d\eq0\pmod n$$
and also
$$\sum_{i\in I_1'\cup J_1}a_i\eq n'r+(n-n')c\eq0\pmod n,$$
where $I_1'=(I_1\sm\{s\})\cup\{t\}$.
Therefore, we must have
\begin{equation}\label{I1J}\sum_{i\in I_1\cup J_1}a_i\eq0\eq\sum_{i\in I_1'\cup J_1}a_i\pmod{n^2}
\end{equation}
and hence $a_s\eq a_t\pmod{n^2}$.

By the last paragraph, there is an integer $r'\eq r\pmod n$ such that $a_i\eq r'\pmod{n^2}$
for all $i\in I$. Similarly, as $1\ls n-n'<n-1$, there is an integer $c'\eq c\pmod n$ such that $a_j\eq c'\pmod{n^2}$
for all $j\in J$. In view of \eqref{I1J}, we have
\begin{equation}
\label{nr} n'r'+(n-n')c'\eq0\pmod{n^2}.
\end{equation}

Take $I_2\se\ I$ with $|I_2|=2n'\ls n$ and $J_2\se J$ with $|J_2|=n-2n'\ls n-1$. Then
$|I_2\cup J_2|=|I_2|+|J_2|=n$ and
$$\sum_{i\in I_2\cup J_2}a_i\eq 2n'r+(n-2n')c\eq 2n\f{r-c}d\eq0\pmod n.$$
Thus, we must have
$$\sum_{i\in I_2\cup J_2}a_i\eq0\pmod{n^2}$$
and hence
\begin{equation}\label{2nr}2n'r'+(n-2n')c'\eq0\pmod{n^2}.
\end{equation}
Combining \eqref{nr} and \eqref{2nr}, we get $nc'\eq0\pmod{n^2}$,
which is impossible since $c'\eq c\not\eq0\pmod n$.

In view of the above, we have completed the proof of Theorem \ref{Th1.2}. \qed

The following lemma in the case $n\ls p$ is just \cite[Lemma 3.4]{G26}.

\begin{lemma} \label{Lem-pn} Let $p$ be an odd prime, and let $a_1,\ldots,a_{p+n-1}\in\Z$, where $n$ is a positive integer. Suppose that
$$\max_{0\ls r\ls p-1}|\{1\ls i\ls p+n-1:\ a_i\eq r\pmod p\}|\ls n.$$ Then, for any $b\in\Z$, we have
$\sum_{i\in I}a_i\eq b\pmod p$ for some $I\se\{1,\ldots,p+n-1\}$
with $|I|=n$.
\end{lemma}
\Proof. Clearly, there is a partition $\{I_s\}_{s=1}^n$ of $\{1,\ldots,p+n-1\}$
such that $|A_s|=|I_s|$ for all $s=1,\ldots,n$, where $A_s=\{a_i+p\Z:\ i\in I_s\}$.
 By the Cauchy-Davenport
theorem, for the sumset $A_1+\cdots+A_n$, we have
$$ |A_1+\cdots+A_n|\gs\min\{p,|A_1|+\cdots+|A_n|-n+1\}=\min\{p,p\}=p$$
and hence $A_1+\cdots+A_n=\Z/p\Z$. So, for each $b\in\Z$ there are $i_s\in I_s$ $(1\ls \ls n)$
such that $\sum_{s=1}^n a_{i_s}\eq b\pmod n$. This concludes the proof. \qed

\medskip
\noindent {\it Proof of Theorem \ref{Th1.3}}.
Assume that the desired result fails. Then, for any $I\se\{1,\ldots,2p+1\}$
with $|I|=p$ and $p\mid\sum_{i\in I}a_i$, we must have $p^2\mid \sum_{i\in I}a_i$.

For each $i=1,\ldots,2p+1$, let us write  $a_i=pq_i+r_i$ with $q_i\in\Z$ and $r_i\in\{1,\ldots,p-1\}$. By the condition, $r_s+r_t\eq r_u+r_v\pmod p$.
We want to show that $a_s+a_t\eq a_u+a_v\pmod {p^2}$ which leads to a contradiction.

{\it Case} 1. Among those $r_i\ (i\in\{1,\ldots,2p+1\}\sm\{s,t,u,v\})$, each can be repeated at most $p-2$ times.

In this case, by Lemma \ref{Lem-pn}, for some $J\se\{1,\ldots,2p+1\}\sm\{s,t,u,v\}$ with $|J|=p-2$,
we have
$$\sum_{j\in J} r_j\eq -(r_s+r_t)\eq-(r_u+r_v)\pmod p.$$
So
$$\sum_{j\in J\cup\{s,t\}}a_j\eq\sum_{j\in J\cup\{s,t\}}r_j\eq0\eq\sum_{j\in J\cup\{u,v\}}r_j
\eq \sum_{j\in J\cup\{u,t\}}a_j\pmod p.$$
Therefore
$$\sum_{j\in J\cup\{s,t\}}a_j\eq0\eq \sum_{j\in J\cup\{u,v\}}a_j\pmod{p^2}$$
and hence $a_s+a_t\eq a_u+a_v\pmod{p^2}$.

{\it Case} 2. $|\{i\in\{1,\ldots,2p+1\}\sm\{s,t,u,v\}:\ r_i=c\}|=p-1$ for some $1\ls c\ls p-1$
and $|\{r_i:\ i\in\{1,\ldots,2p+1\}\sm\{s,t,u,v\}\}|=2$.

Take $I\se\{1,\ldots,2p+1\}\sm\{s,t,u,v\}$ with $|I|=p-1$ such that $r_i=c$ for all $i\in I$.
Suppose that
$r_j=d$ for all $j\in\{1,\ldots,2p+1\}\sm(I\cup\{s,t,u,v\})$.
Note that $d\not=c$.

 Choose the unique integer $k\in\{0,\ldots,p-1\}$ such that $k(c-d)\eq 2d-(r_s+r_t)\pmod p$.
If $r_s+r_t\not\eq c+d\pmod p$, then $k\not=p-1$, and for any $I_1\se I$ and $I_2\se\{1,\ldots,2p+1\}\sm(I\cup\{s,t,u,v\})$ with $|I_1|=k$ and $|I_2|=p-2-k$, we have
$$\sum_{i\in I_1\cup I_2}a_i\eq \sum_{i\in I_1\cup I_2}r_i\eq kc+(p-2-k)d\eq-(r_s+r_t)\eq-(r_u+r_v)\pmod p$$
and hence
$$\sum_{i\in I_1\cup I_2\cup\{s,t\}}a_i\eq0\eq \sum_{i\in I_1\cup I_2\cup\{u,v\}}a_i\pmod{p^2}$$
which implies that $a_s+a_t\eq a_u+a_v\pmod{p^2}$.

Now we consider the remaining case $r_s+r_t\eq c+d\pmod p$.
Choose $I_1\se I$ with $|I_1|=p-2$, and also take $I_2\se\{1,\ldots,2p+1\}\sm (I\cup\{s,t,u,v\})$
with $|I_2|=p-2$. Then $J=I_1\cup I_2\cup\{u,v\}$ has cardinality $2p-2$. Note that $r_u,r_v\not=c$
since $|\{1\ls i\ls 2p+1:\ r_i=c\}|<p$. As $r_u+r_v\eq c+d\pmod p$, we also have $r_u,r_v\not=d$. Thus $\{r_u,r_v\}\cap\{c,d\}=\emptyset$. Choose the unique $k\in\{0,\ldots,p-1\}$
 with
 \begin{equation}\label{d-c}
 (k+1)(d-c)\eq c-r_v\pmod p.
 \end{equation}
  As $r_v\not=c,d$, we have $0\ls k\ls p-3$.
 Choose $i_1\in I_1$, $i_2\in I_2$, $J_1\se I_1\sm\{i_1\}$ and $J_2\se I_2\sm{i_2}$ with $|J_1|=p-3-k$ and $|J_2|=k$. Then $J=J_1\cup J_2\cup\{v\}$
 has cardinality $p-2$, and
 $$\sum_{j\in J}r_j=r_v+(p-3-k)c+kd\eq-(c+d)\eq-(r_s+r_t)\pmod p$$
 in view of \eqref{d-c}. Thus
 $$\sum_{j\in J\cup\{s,t\}}r_j\eq0\eq\sum_{j\in J\cup\{i_1,i_2\}}r_j\pmod p$$
 and hence
 $$\sum_{j\in J\cup\{s,t\}}a_j\eq0\eq\sum_{j\in J\cup\{i_1,i_2\}}a_j\pmod {p^2},$$
 which yields that $a_s+a_t\eq a_{i_1}+a_{i_2}\pmod{p^2}$.
 Similarly, we can prove that $a_u+a_v\eq a_{i_1}+a_{i_2}\pmod{p^2}$.
 Thus $a_s+a_t\eq a_u+a_v\pmod{p^2}$ as desired.

In view of the above, we have completed the proof of Theorem \ref{Th1.3}.
\qed

\Ack. The author would like to thank Prof. Qinghai Zhong for providing some references.

\end{document}